\documentclass[11pt,a4paper]{article}
\usepackage[english]{babel}
\usepackage{times}
\usepackage{graphicx}
\usepackage{amscd}
\usepackage{amsmath}
\usepackage{amsfonts}
\usepackage{amssymb}
\usepackage{amsthm}
\usepackage{latexsym}
\newtheorem{theorem}{Theorem}

\newtheorem{lemma}[theorem]{Lemma}

\newtheorem{remark}[theorem]{Remark}

\def\neweq#1{\begin{equation}\label{#1}}
\def\endeq{\end{equation}}
\def\eq#1{(\ref{#1})}

\newcommand{\R}{\mathbb{R}}
\newcommand{\N}{\mathbb{N}}
\newcommand{\HH}{\mathcal{H}}
\newcommand{\eps}{\varepsilon}

\newcommand{\integ}{\int_\Omega}

\begin{document}

\title{Modeling suspension bridges\\
through the von {K}\'arm\'an quasilinear plate equations}

\author{Filippo GAZZOLA - Yongda WANG\\
{\small Dipartimento di Matematica, Politecnico di Milano (Italy)}}
\date{}
\maketitle

\begin{center}
{\em Dedicated to Djairo Guedes de Figueiredo, on the occasion of his 80th birthday.}
\end{center}

\begin{abstract}
A rectangular plate modeling the deck of a suspension bridge is considered. The plate may widely oscillate, which suggests to consider models from
nonlinear elasticity. The von K\'arm\'an plate model is studied, complemented with the action of the hangers and with suitable boundary conditions
describing the behavior of the deck. The oscillating modes are determined in full detail. Existence and multiplicity of static equilibria are then
obtained under different assumptions on the strength of the buckling load.\par
\end{abstract}

\section{Introduction and motivations: nonlinear behavior of suspension bridges}

The purposes of the present paper are to set up a nonlinear model to describe the static behavior of a suspension bridge and to study possible
multiplicity of the equilibrium positions. We view the deck of the bridge as a long narrow rectangular thin plate, hinged on its short edges
where the bridge is supported by the ground, and free on its long edges. Let $L$ denote its length and $2\ell$ denote its width; a realistic
assumption is that $2\ell\cong\frac{L}{100}$.

The rectangular plate resists to transverse loads exclusively by means of bending. The flexural properties of a plate strongly depend on its thickness,
which we denote by $d$, compared with its width $2\ell$ and its length $L$. We assume here that $2\ell<L$ so that $d$ is to be compared with $2\ell$.
From Ventsel-Krauthammer \cite[$\S$ 1.1]{ventsel} we learn that plates may be classified according to the ratio $2\ell/d$:

$\bullet$ if $2\ell\le8d$ we have a thick plate and the analysis of these plates includes all the components of stresses, strains and displacements
as for solid three-dimensional bodies;\par
$\bullet$ if $8d\le2\ell\le80d$ we have a thin plate which may behave in both linear and nonlinear regime according to how large is the ratio between
its deflection and its thickness $d$;\par
$\bullet$ if $2\ell\ge80d$ the plate behaves like a membrane and lacks of flexural rigidity.\par
Let us now turn to a particular suspension bridge. The main span of the collapsed Tacoma Narrows Bridge \cite{ammann,tacoma} had the measures
\neweq{lengths}
L=2800\, \mbox{ft.}\, ,\qquad2\ell=39\, \mbox{ft.}\, ,\qquad d=4\, \mbox{ft.}\ ,
\endeq
see p.11 and Drawings 2 and 3 in \cite{ammann}. Therefore, $2\ell/d=9.75$ and
\begin{center}
{\bf the deck of the Tacoma Narrows Bridge may be considered as a thin plate}.
\end{center}
It is clear that modern suspension bridges with their stiffening trusses are more similar to thick plates.\par
Which theory (linear or nonlinear) models a thin plate depends on the magnitude $W$ of its maximal deflection. If we denote again by $d$ its
thickness, two cases may occur, according to Ventsel-Krauthammer \cite[$\S$ 1.1]{ventsel}:\par
$\bullet$ if $W/d\le0.2$ the plate is classified as stiff: these plates carry loads two dimensionally, mostly by internal bending, twisting
moments and by transverse shear forces;\par
$\bullet$ if $W/d\ge0.3$ the plate is classified as flexible: in this case, the deflections will be accompanied
by stretching of the surface.\par\smallskip
A fundamental feature of stiff plates is that the equation of static equilibrium for a plate
element may be set up for an original (undeformed) configuration of the plate: in this case a linear theory describes with sufficient
accuracy the behavior of the plate. Flexible plates behave somehow in between membranes and stiff plates: when $W\gg d$ the membrane
action is dominant and the flexural stress can be neglected compared with the membrane stress: in this case, a linear theory is not
enough to describe accurately the behavior of the plate and one has to stick to nonlinear theories.\par According to Scott
\cite[pp.49-51]{scott} (see also \cite[p.60]{ammann} and the video \cite{tacoma}), the Board of Engineers stated that under pure
longitudinal oscillations {\em ...the lateral deflection of the center bridge was not measured but did not appear excessive, perhaps
four times the width of the yellow center line (about 2 ft.)} while, after the appearance of the torsional oscillation, {\em ...the
roadway was twisting almost $45^\circ$ from the horizontal, with one side lurching 8.5 m.\ above the other}. This means that it was $W=$2
ft.\ during the vertical oscillations without torsion and $W=$14 ft.\ when the torsional oscillation appeared at the Tacoma Narrows
Bridge. In view of \eq{lengths}, we then have $W/d=0.5$ under pure longitudinal oscillations and $W/d=3.5$ in presence of torsional
oscillations. The conclusion is that
\begin{center}
{\bf the Tacoma Narrows Bridge oscillated in a nonlinear regime}.
\end{center}

This was already known to civil engineers about half a century ago (see e.g.\ \cite{west}) although the difficulties in tackling nonlinear models
prevented a systematic study of the nonlinear regimes. In recent years, the necessity of nonlinear models became even more evident
\cite{brown,gazz,lacarbonara,plautdavis} and the progress of tools in nonlinear analysis and in numerics gives the chance to obtain
responses from nonlinear models. Which nonlinear model should be used is questionable. For two different models of ``nonlinear degenerate bridges''
a structural instability has been recently highlighted in \cite{arga,bergaz}, both numerically and theoretically: it is shown that the torsional
instability has a structural origin and not a mere aerodynamic justification as usually assumed in engineering literature, see
\cite[Section 12]{pugsley} and \cite{scanlan,scantom}. By ``degenerate'' bridge we mean that the deck is not modeled through a full plate as in
actual bridges.\par
A first interesting linear plat theory is due to Kirchhoff \cite{Kirchhoff} in 1850, but it was only 60 years later (in 1910)
that von K\'arm\'an \cite{karman} suggested a two-dimensional system in order to describe large deformations of a thin plate. This theory was considered
a breakthrough in several scientific communities, including in the National Advisory Committee for Aeronautics, an American federal agency during the
19th century: the purpose of this agency was to undertake, to promote, and to institutionalize aeronautical research and the von K\'arm\'an equations
were studied for a comparison between theoretical and experimental results, see \cite{levy2,levy3}. In his report, Levy \cite{levy2} writes that
{\em In the design of thin plates that bend under lateral and edge loading, formulas based on the Kirchhoff theory which neglects stretching and
shearing in the middle surface are quite satisfactory provided that the deflections are small compared with the thickness. If deflections are of
the same order as the thickness, the Kirchhoff theory may yield results that are considerably in error and a more rigorous theory that takes account of
deformations in the middle surface should therefore be applied. The fundamental equations for the more exact theory have been derived by von K\'arm\'an}.\par
In order to describe its structural behavior, in this paper we view the bridge deck as a plate subject to the restoring force due to the hangers
and behaving nonlinearly: we adapt the quasilinear von K\'arm\'an \cite{karman} model to a suspension bridge. In spite of the fact that this model
received severe criticisms about its physical soundness (see \cite[pp.601-602]{truesdell}), many authors have studied the von K\'arm\'an system,
see our incomplete bibliography.
In particular, Ciarlet \cite{ciarlet} provides an important justification of the von K\'arm\'an equations. He makes an asymptotic expansion with respect to
the thickness of a three-dimensional class of elastic plates under suitable loads. He then shows that the leading term of the expansion solves a system of
equations equivalent to those of von K\'arm\'an. Davet \cite{davet} pursues further and
proves that the von K\'arm\'an equations may be justified by asymptotic expansion methods starting from very general 3-dimensional constitutive laws.\par
Following the setting in \cite{fergaz} (see also \cite{algwaiz,yongda,yongda2}), we consider a thin and narrow rectangular plate $\Omega$ where
the two short edges are assumed to be hinged whereas the two long edges are assumed to be free. The plate is subject to three actions:\par
$\bullet$ normal dead and live loads acting orthogonally on the plate;\par
$\bullet$ edge loading, also called buckling loads, namely compressive forces along its edges;\par
$\bullet$ the restoring force due to the hangers, which acts in a neighborhood of the long edges.\par\smallskip
The simplest action is the first one: the dead load is the structural weight whereas the live load may be a wind gust or some vehicle going through the
bridge. As already pointed out by von K\'arm\'an \cite{karman}, large edge loading may yield buckling, that is, the plate may deflect out of its plane
when these forces reach a certain magnitude. The edge loading is called prestressing in engineering literature, see \cite{menn}. This was mathematically
modeled by Berger \cite{berger} with a suitable nonlocal term and tackled with variational methods in a recent paper \cite{algwaiz} which shows that
large prestressing leads to buckling, that is, multiplicity of solutions of the corresponding equation. The critical buckling load may be computed by
finding the smallest eigenvalue of an associated linear problem.\par
An important contribution of Berger-Fife \cite{bergerfife} reduces the von K\'arm\'an system to a variational problem
and tackles it with critical point and bifurcation theories (we point out that there are two different authors named Berger in our references).
Subsequently, Berger \cite{bergervk} made a full analysis of the unloaded clamped plate problem (Dirichlet boundary conditions) which is somehow
the simplest one but does not model the physical situation of a bridge. The loaded clamped plate was analyzed in \cite{knight,knight2} where
existence and possible nonuniqueness results were obtained. Different boundary conditions for the hinged plate (named after Navier) and
for free boundaries were then analyzed with the same tools by Berger-Fife \cite{bergerfife2}. Since free edges of the plate are considered, this last paper
is of particular interest for our purposes. As clearly stated by Ciarlet \cite[p.353]{ciarlet} the boundary conditions for the Airy function {\em are
often left fairly vague in the literature}; we take them in a ``dual form'', that is, more restrictions for the edges yield less restrictions for the
Airy function and viceversa.\par
We adapt here these plate models to a suspension bridge. The main novelties are that the function representing the vertical
displacement of the rectangular plate $\Omega$ satisfies a mixed hinged and free boundary conditions and that the restoring force due to the hangers is
taken into account. It is well-known \cite{figueiredo} that nonlinear elliptic systems are fairly delicate to tackle with variational methods.
The model describing a suspension bridge involves a fourth order quasilinear elliptic system and this brings further difficulties, in particular in the
definition of the action functional. We start by setting in full detail the linear theory which enables us to determine the critical prestressing values
leading to buckling and to the multiplicity of solutions. Then we analyze the problem with normal dead loads but no restoring force and we obtain results in
the spirit of \cite{bergervk,bergerfife}. Finally, we introduce the restoring force due to the hangers and we prove existence and multiplicity of the
equilibrium positions.

\section{Functional framework and the quasilinear equations}

\subsection{Elastic energies of a plate}

The bending energy of the plate $\Omega$ involves curvatures of the surface. Let $\kappa_1$ and $\kappa_2$ denote the principal
curvatures of the graph of the (smooth) function $u$ representing the vertical displacement of the plate in the downwards direction, then
the Kirchhoff model \cite{Kirchhoff} for the bending energy of a deformed plate $\Omega$ of thickness $d>0$ is
\neweq{curva}
\mathbb{E}_B(u)=\frac{E\, d^3}{12(1-\sigma^2)}\int_\Omega\left(\frac{\kappa_1^2}{2}+\frac{\kappa_2^2}{2}+\sigma\kappa_1\kappa_2\right)\, dxdy
\endeq
where $\sigma$ is the Poisson ratio defined by $\sigma=\frac{\lambda}{2\left(\lambda+\mu\right)}$ and $E$ is the Young modulus defined by
$E=2\mu(1+\sigma)$, with the so-called Lam\'e constants $\lambda,\mu $ that depend on the material. For physical reasons it holds that $\mu >0$ and
usually $\lambda>0$ so that
\neweq{sigma}
0<\sigma<\frac{1}{2}.
\endeq

For small deformations the terms in (\ref{curva}) are taken as approximations being purely quadratic with respect to the second
order derivatives of $u$. More precisely, for small deformations $u$, one has
\neweq{smalldeformations}
(\kappa_1+\kappa_2)^2\approx(\Delta u)^2\ ,\quad\kappa_1\kappa_2\approx\det(D^2u)=u_{xx}u_{yy}-u_{xy}^{2}\ ,
\endeq
and therefore
$$\frac{\kappa_1^2}{2}+\frac{\kappa_2^2}{2}+\sigma\kappa_1\kappa_2\approx\frac{1}{2}(\Delta u)^2+(\sigma-1)\det(D^2u).$$
Then, if $f$ denotes the external vertical load (including both dead and live loads) acting on the plate $\Omega$ and if $u$ is the corresponding (small)
vertical displacement of the plate, by \eq{curva} we have that the total energy $\mathbb E_T$ of the plate becomes
\begin{eqnarray}
\mathbb E_T(u) &=&\mathbb{E}_B(u)-\int_\Omega fu\, dxdy \label{eq:original-E}\\
&=&\frac{E\, d^3}{12(1-\sigma^2)}\int_{\Omega}\left(\frac{1}{2}\left(\Delta u\right)^{2}
-(1-\sigma)\det(D^2u)\right) \, dxdy-\int_\Omega fu\, dxdy.\notag
\end{eqnarray}
Note that the ``quadratic'' functional $\mathbb{E}_B(u)$ is positive whenever $|\sigma|<1$, a condition which is ensured by \eq{sigma}.\par
If large deformations are involved, one does not have a linear strain-displacement relation resulting in \eq{smalldeformations}. For a plate of uniform
thickness $d>0$, one assumes that the plate has a middle surface midway between its parallel faces that, in equilibrium, occupies the region $\Omega$ in
the plane $z=0$. Let $w=w(x,y)$, $v=v(x,y)$, $u=u(x,y)$ denote the components (respectively in the $x$, $y$, $z$ directions) of the displacement vector
of the particle of the middle surface which, when
the plate is in equilibrium, occupies the position $(x,y)\in\Omega$: $u$ is the component in the vertical $z$-direction which is related to bending while
$w$ and $v$ are the in-plane stretching components. For large deformations of $\Omega$ there is a coupling between $u$ and $(w,v)$. In order to describe
it, we compute the stretching in the $x$ and $y$ directions (see e.g.\ \cite[(7.80)]{ventsel}):
\neweq{approxstretch}
\eps_x=\sqrt{1+2w_x+u_x^2}-1\approx w_x+\frac{u_x^2}{2}\, ,\quad\eps_y=\sqrt{1+2v_y+u_y^2}-1\approx v_y+\frac{u_y^2}{2}
\endeq
where the approximation is due to the fact that, compared to unity, all the components are small in the horizontal directions $x$ and $y$.
One can also compute the shear strain (see e.g.\ \cite[(7.81)]{ventsel}):
\neweq{sstrain}
\gamma_{xy}\approx w_y+v_x+u_x u_y\, .
\endeq
Finally, it is convenient to introduce the so-called stress resultants which are the integrals of suitable components of the strain
tensor (see e.g.\ \cite[(1.22)]{lagnese}), namely,
$$N^x=\frac{Ed}{1-\sigma^2}\left(w_x+\sigma v_y+\frac12 u_x^2+\frac{\sigma}{2}u_y^2\right)\, ,\quad
N^y=\frac{Ed}{1-\sigma^2}\left(v_y+\sigma w_x+\frac12 u_y^2+\frac{\sigma}{2}u_x^2\right)\, ,$$
\neweq{sresultant}
N^{xy}=\frac{Ed}{2(1+\sigma)}\left(w_y+v_x+u_x u_y\right)\, ,
\endeq
so that
$$\eps_x=\frac{N^x-\sigma N^y}{Ed}\, ,\quad\eps_y=\frac{N^y-\sigma N^x}{Ed}\, ,\quad\gamma_{xy}=\frac{2(1+\sigma)}{Ed}\, N^{xy}\, .$$

We are now in a position to define the energy functional. The first term $\mathbb E_T(u)$ of the energy is due to pure bending and to external loads
and was already computed in \eq{eq:original-E}. For large deformations, one needs to consider also the interaction with the stretching components $v$
and $w$ and the total energy reads (see \cite[(1.7)]{llagnese})
\neweq{J}
J(u,v,w)=\mathbb E_T(u)+\frac{E\, d}{2(1-\sigma^2)}\int_\Omega\left(\eps_x^2+\eps_y^2+2\sigma\, \eps_x\eps_y+\frac{1-\sigma}{2}\, \gamma_{xy}^2\right)\, dxdy
\endeq
which has to be compared with \eq{eq:original-E}. In view of
\eq{approxstretch}-\eq{sstrain} the additional term $I:=J-\mathbb E_T$ may also be written as
\begin{eqnarray*}
I(u,v,w) &\!=\!& \frac{E\, d}{2(1\!-\!\sigma^2)}\int_\Omega\!\left\{\left(w_x\!+\!\frac{u_x^2}{2}\right)^2\!+\!\left(v_y\!+\!\frac{u_y^2}{2}\right)^2\!
+\!2\sigma\left(w_x\!+\!\frac{u_x^2}{2}\right)\left(v_y\!+\!\frac{u_y^2}{2}\right)\right\}dxdy \\
\ & \ &+\frac{E\, d}{4(1\!+\!\sigma)}\int_\Omega\!(w_y\!+\!v_x\!+\!u_xu_y)^2dxdy\, .
\end{eqnarray*}

The next step is to derive the equations and boundary conditions which characterise the critical points of $J$; this will be done in the two following
subsections.

\subsection{The Euler-Lagrange equation}

Let $L$ denote the length of the plate $\Omega$ and $2\ell$ denote its width with $2\ell\cong\frac{L}{100}$. In order to simplify the Fourier series expansions
we take $L=\pi$ so that, in the sequel,
$$
\Omega=(0,\pi)\times(-\ell,\ell)\subset\R^2\qquad\mbox{(with $\ell\ll\pi$).}
$$

The natural functional space where to set up the problem is
$$H^2_*(\Omega):=\Big\{w\in H^2(\Omega);\, w=0\mbox{ on }\{0,\pi\}\times (-\ell,\ell)\Big\}\ .$$
We also define
$$\HH_*(\Omega):=\mbox{ the dual space of }H^2_*(\Omega)$$
and we denote by $\langle\cdot,\cdot\rangle$ the corresponding duality. Since we are in the plane, $H^2(\Omega)\subset
C^0(\overline{\Omega})$ so that the condition on $\{0,\pi\}\times(-\ell,\ell)$ introduced in the definition of $H^2_*(\Omega)$ makes sense.
On the space $H^2(\Omega)$ we define the Monge-Amp\`ere operator
\neweq{MA}
[\phi,\psi]:=\phi_{xx}\psi_{yy}+\phi_{yy}\psi_{xx}-2\phi_{xy}\psi_{xy}\qquad\forall\phi,\psi\in H^2(\Omega)
\endeq
so that, in particular, $[\phi,\phi]=2{\rm det}(D^2\phi)$ where $D^2\phi$ is the Hessian matrix of $\phi$.\par
As pointed out in \cite[Lemma 4.1]{fergaz}, $H^2_*(\Omega)$ is a Hilbert space when endowed with the scalar product
$$(u,v)_{H^2_*(\Omega)}:=\int_\Omega\Big(\Delta u\Delta v-(1-\sigma)[u,v]\Big)\, dxdy \, .$$
The corresponding norm then reads
$$\|u\|_{H^2_*(\Omega)}:=\left(\int_\Omega\Big(|\Delta u|^2-(1-\sigma)[u,u]\Big)\, dxdy \right)^{1/2}\, .$$
The unique minimiser $u$ of the convex functional $\mathbb{E}_T$ in \eq{eq:original-E} over the space $H^2_*(\Omega)$ satisfies the Euler-Lagrange equation
\neweq{loadpb}
\frac{E\, d^3}{12(1-\sigma^2)}\, \Delta^2 u=f(x,y)\qquad\mbox{in }\Omega\, .
\endeq

On the other hand, the Euler-Lagrange equation for the energy $J$ in \eq{J} characterises the critical points of $J$: we need to compute the variation
$\delta J$ of $J$ and to find triples $(u,v,w)$ such that
$$\langle\delta J(u,v,w),(\phi,\psi,\xi)\rangle=\lim_{t\to0}\frac{J(u+t\phi,v+t\psi,w+t\xi)-J(u,v,w)}{t}=0\qquad\forall\phi,\psi,\xi\in C^\infty_c(\Omega)\, .$$
\vfill\eject
After replacing $N^x$, $N^y$, $N^{xy}$, see \eq{sresultant}, this yields
$$\frac{E\, d^3}{12(1-\sigma^2)}\int_\Omega \big(\Delta u\Delta\phi+(\sigma-1)[u,\phi]\big)\, dxdy$$
$$+\int_\Omega\big((N^xu_x+N^{xy}u_y)\phi_x+(N^yu_y+N^{xy}u_x)\phi_y\big)\, dxdy=\int_\Omega f\phi\, dxdy\qquad\forall\phi\in C^\infty_c(\Omega)$$
$$\int_\Omega\big(N^y\psi_y+N^{xy}\psi_x\big)\, dxdy=0\qquad\forall\psi\in C^\infty_c(\Omega)$$
$$\int_\Omega\big(N^x\xi_x+N^{xy}\xi_y\big)\, dxdy =0\qquad\forall\xi\in C^\infty_c(\Omega).$$
Thanks to some integration by parts and by arbitrariness of the test functions, we may rewrite the above identities in strong form
\neweq{systzero}
\begin{array}{cc}
\frac{E\, d^3}{12(1-\sigma^2)}\Delta^2u-(N^xu_x+N^{xy}u_y)_x-(N^yu_y+N^{xy}u_x)_y=f\qquad\mbox{in }\Omega\, ,\\
N^y_y+N^{xy}_x=0\, ,\quad N^x_x+N^{xy}_y=0\qquad\mbox{in }\Omega\, .
\end{array}\endeq
The last two equations in \eq{systzero} show that there exists a function $\Phi$ (called Airy stress function), unique up to an affine function, such that
\neweq{Phi}
\Phi_{yy}=N^x,\quad \Phi_{xx}=N^y,\quad \Phi_{xy}=-N^{xy}\, .
\endeq
Then, after some tedious computations, by using the Monge-Amp\`ere operator \eq{MA} and by normalising the coefficients, the system \eq{systzero} may be written as
\neweq{mostro}
\left\{\begin{array}{ll}
\Delta^2\Phi=-[u,u]\quad & \mbox{in }\Omega\\
\Delta^2 u=[\Phi,u]+f\quad & \mbox{in }\Omega\, .
\end{array}\right.
\endeq

In a plate subjected to compressive forces along its edges, one
should consider a prestressing constraint which may lead to
buckling. Then the system \eq{mostro} becomes
\neweq{mostro1}
\left\{\begin{array}{ll}
\Delta^2\Phi=-[u,u]\quad & \mbox{in }\Omega\\
\Delta^2 u=[\Phi,u]+f+\lambda[F,u]\quad & \mbox{in }\Omega\, .
\end{array}\right.
\endeq
The term $\lambda[F,u]$ in the right hand side of \eq{mostro1}
represents the boundary stress. The parameter $\lambda\ge0$ measures
the magnitude of the compressive forces acting on $\partial\Omega$
while the smooth function $F$ satisfies
\neweq{F}
F\in C^4(\overline{\Omega})\, ,\quad\Delta^2F=0\mbox{ in }\Omega\,
,\quad F_{xx}=F_{xy}=0\mbox{ on }(0,\pi)\times\{\pm\ell\}\, ,
\endeq
see \cite[pp.228-229]{bergerfife2}: the term $\lambda F$ represents
the stress function in the plate resulting from the applied force if
the plate were artificially prevented from deflecting and the
boundary constraints in \eq{F} physically mean that no external
stresses are applied on the free edges of the plate. Following
Knightly-Sather \cite{knight3}, we take
\neweq{choiceF}
F(x,y)=\frac{\ell^2-y^2}{2}\qquad\mbox{so that}\quad[F,u]=-u_{xx}\, .
\endeq
Therefore, \eq{mostro1} becomes
\neweq{mostro11}
\left\{\begin{array}{ll}
\Delta^2\Phi=-[u,u]\quad & \mbox{in }\Omega\\
\Delta^2 u=[\Phi,u]+f-\lambda u_{xx}\quad & \mbox{in }\Omega\, .
\end{array}\right.
\endeq

\subsection{Boundary conditions}

We now determine the boundary conditions to be associated to
\eq{mostro11}. In literature these equations are usually considered
under Dirichlet boundary conditions, see \cite[$\S$ 1.5]{ciarletvk}
and \cite[p.514]{villaggio}. But since we aim to model a suspension
bridge, these conditions are not the correct ones. Following
\cite{fergaz} (see also \cite{algwaiz,yongda}) we view the deck
of a suspension bridge as a long narrow rectangular thin plate
hinged at its two opposite short edges and free on the remaining two long edges.\par
Let us first consider the two short edges $\{0\}\times(-\ell,\ell)$ and $\{\pi\}\times(-\ell,\ell)$. Due to the connection with the ground, $u$ is assumed
to be hinged there and hence it satisfies the Navier boundary conditions:
\neweq{bcuno}
u=u_{xx}=0\qquad\mbox{on }\{0,\pi\}\times(-\ell,\ell)\, .
\endeq
In this case, Ventsel-Krauthammer \cite[Example 7.4]{ventsel} suggest that $N^x=v=0$ on $\{0,\pi\}\times(-\ell,\ell)$. In view of \eq{sresultant} this yields
$$0=w_x+\sigma v_y+\frac12 u_x^2+\frac{\sigma}{2}u_y^2=w_x+\frac12 u_x^2=\frac{Ed}{(1-\sigma^2)\sigma}N^y$$
where the condition $u_y=0$ comes from the first of \eq{bcuno}. In
turn, by \eq{Phi} this implies that $\Phi_{xx}=0$ on
$\{0,\pi\}\times(-\ell,\ell)$. For the second boundary condition we
recall that $N^x=0$ so that, by \eq{Phi}, also $\Phi_{yy}=0$: since
the Airy function $\Phi$ is defined up to the addition of an affine
function, we may take $\Phi=0$. Summarising, we also have
\neweq{bcdue}
\Phi=\Phi_{xx}=0\qquad\mbox{on }\{0,\pi\}\times(-\ell,\ell)\, .
\endeq

On the long edges $(0,\pi)\times\{\pm\ell\}$ the plate is free, which results in
\neweq{bctre}
u_{yy}+\sigma u_{xx}=u_{yyy}+(2-\sigma)u_{xxy}=0\qquad\mbox{on }(0,\pi)\times\{\pm\ell\}\, ,
\endeq
see e.g.\ \cite[(2.40)]{ventsel} or \cite{fergaz}. Note that here
the boundary conditions do not depend on $\lambda$. For the Airy
stress function $\Phi$, we follow the usual Dirichlet boundary
condition on $(0,\pi)\times\{\pm\ell\}$, see \cite{bergerfife,bergerfife2}. Then
\neweq{bcquattro}
\Phi=\Phi_y=0\qquad\mbox{on }(0,\pi)\times\{\pm\ell\}\, .
\endeq

These boundary conditions suggest to introduce the following subspace of $H_*^2(\Omega)$
$$H_{**}^2(\Omega):=\{u \in H_*^2(\Omega): u=u_y=0 \mbox{ on }(0,\pi)\times \{\pm \ell\}\},$$
which is a Hilbert space when endowed with the scalar product and norm
$$(u,v)_{H_{**}^2(\Omega)}:=\int_{\Omega}{\Delta u\Delta v}\, ,\qquad \|u\|_{H_{**}^2(\Omega)}:=\left(\int_{\Omega}{|\Delta u|^2}\right)^{1/2}.$$
We denote the dual space of $H_{**}^2(\Omega)$ by $\mathcal{H}_{**}(\Omega)$.

\subsection{The quasilinear von K\'{a}rm\'{a}n equations modeling suspension bridges}

By putting together the Euler-Lagrange equation \eq{mostro11} and the boundary conditions \eq{bcuno}-\eq{bcquattro} we obtain the system
\neweq{mostro2}
\left\{\begin{array}{ll}
\Delta^2\Phi=-[u,u]\quad & \mbox{in }\Omega\\
\Delta^2 u=[\Phi,u]+f-\lambda u_{xx}\quad & \mbox{in }\Omega\\
u=\Phi=u_{xx}=\Phi_{xx}=0\quad & \mbox{on }\{0,\pi\}\times(-\ell,\ell)\\
u_{yy}+\sigma u_{xx}=u_{yyy}+(2-\sigma)u_{xxy}=0\quad & \mbox{on }(0,\pi)\times\{\pm\ell\}\\
\Phi=\Phi_y=0\quad & \mbox{on }(0,\pi)\times\{\pm\ell\}\, .
\end{array}\right.
\endeq

In a plate modeling a suspension bridge, one should also add the nonlinear restoring action due to the hangers. Then the second
equation in \eq{mostro2} becomes
\neweq{nuova}
\Delta^2 u+\Upsilon(y)g(u)=[\Phi,u]+f-\lambda u_{xx}\quad\mbox{in }\Omega\, .
\endeq
Here $\Upsilon$ is the characteristic function of $(-\ell,-\ell+\eps)\cup (\ell-\eps,\ell)$ for some small $\eps$. This means that the
restoring force due to the hangers is concentrated in two tiny parallel strips adjacent to the long edges (the free part of the boundary). The Official
Report \cite[p.11]{ammann} states that the region of interaction of the hangers with the plate was of approximately 2 ft on each side: this means that
$\eps\approx\frac{\pi}{1500}$. Augusti-Sepe \cite{sepe1} (see also \cite{sepe2}) view the restoring force at the endpoints of a cross-section of
the deck as composed by two connected springs, the top one representing the action of the sustaining cable and the bottom one (connected with the deck)
representing the hangers. And the action of the cables is considered by Bartoli-Spinelli \cite[p.180]{bartoli} the main cause of the nonlinearity of the
restoring force: they suggest quadratic and cubic perturbations of a linear behavior. Assuming that {\bf the vertical axis is oriented downwards}, the restoring
force acts in those parts of the deck which are below the equilibrium position (where $u>0$) while it exerts no action where the deck is above the
equilibrium position ($u<0$). Taking into account all these facts, for the explicit action of the restoring force, we take
\neweq{g}
g(u)=(ku+\delta u^3)^+
\endeq
which is a compromise between the nonlinearities suggested by McKenna-Walter \cite{McKennaWalter} and Plaut-Davis \cite{plautdavis} and follows the
idea of Ferrero-Gazzola \cite{fergaz}. Here $k>0$ denotes the Hooke constant of elasticity of steel (hangers) while $\delta>0$ is a small parameter
reflecting the nonlinear behavior of the sustaining cables. Only the positive part is taken into account due to possible slackening, see
\cite[V-12]{ammann}: the hangers behave as a restoring force if extended (when $u>0$) and give no contribution when they lose tension (when $u\le0$).\par
By assuming \eq{g}, and inserting \eq{nuova} into \eq{mostro2} leads
to the problem
\neweq{finalmostro}
\left\{\begin{array}{ll}
\Delta^2\Phi=-[u,u]\quad & \mbox{in }\Omega\\
\Delta^2 u+\Upsilon(y)(ku+\delta u^3)^+=[\Phi,u]+f-\lambda u_{xx}\quad & \mbox{in }\Omega\\
u=\Phi=u_{xx}=\Phi_{xx}=0\quad & \mbox{on }\{0,\pi\}\times(-\ell,\ell)\\
u_{yy}+\sigma u_{xx}=u_{yyy}+(2-\sigma)u_{xxy}=0\quad & \mbox{on }(0,\pi)\times\{\pm\ell\}\\
\Phi=\Phi_y=0\quad & \mbox{on }(0,\pi)\times\{\pm\ell\}\, .
\end{array}\right.
\endeq

Finally, we go back to the original unknowns $u$, $v$, $w$. After that a solution $(u,\Phi)$ of \eq{mostro2} or \eq{finalmostro} is found,
\eq{sresultant}-\eq{Phi} yield
$$w_x+\sigma v_y=\frac{1-\sigma^2}{E\, d}\, \Phi_{yy}-\frac12 u_x^2-\frac{\sigma}{2}u_y^2\ ,\quad
\sigma w_x+v_y=\frac{1-\sigma^2}{E\, d}\, \Phi_{xx}-\frac12 u_y^2-\frac{\sigma}{2}u_x^2$$
which immediately gives $w_x$ and $v_y$. Upon integration, this gives $w=w(x,y)$ up to the addition of a function only depending on $y$ and $v=v(x,y)$ up
to the addition of a function depending only on $x$. These two additive functions are determined by solving the last constraint given by \eq{sresultant}-\eq{Phi},
that is,
$$w_y+v_x=-\frac{2(1+\sigma)}{E\, d}\, \Phi_{xy}-u_x-u_y\, .$$

\section{Main results}

With no further mention, we assume \eq{sigma}. The first step to study \eq{mostro2} and \eq{finalmostro} is to analyze the spectrum of
the linear problem obtained by taking $\Phi=f=k=\delta=0$:
\neweq{linearproblem}
\left\{\begin{array}{ll}
\Delta^2 u+\lambda u_{xx}=0\quad & \mbox{in }\Omega\\
u=u_{xx}=0\quad & \mbox{on }\{0,\pi\}\times(-\ell,\ell)\\
u_{yy}+\sigma u_{xx}=u_{yyy}+(2-\sigma)u_{xxy}=0\quad & \mbox{on }(0,\pi)\times\{\pm\ell\}\, .
\end{array}\right.
\endeq

In Section \ref{prspect} we prove the following result

\begin{theorem}\label{spectrum}
The problem \eqref{linearproblem} admits a sequence of divergent
eigenvalues
$$
\lambda_1<\lambda_2\le...\le\lambda_k\le...
$$
whose corresponding eigenfunctions $\{\overline{e}_k\}$ form a complete orthonormal system in $H^2_*(\Omega)$.\par Moreover, the
least eigenvalue $\lambda_1$ is simple and is the unique value of $\lambda\in((1-\sigma)^2,1)$ such that
$$\sqrt{1-\lambda^{1/2}}\, \big(\lambda^{1/2}+1-\sigma\big)^2\tanh(\ell\sqrt{1-\lambda^{1/2}}\, )=\sqrt{1+\lambda^{1/2}}\,
\big(\lambda^{1/2}-1+\sigma\big)^2\tanh(\ell\sqrt{1+\lambda^{1/2}}\, )\, ;$$ the corresponding eigenspace is generated by the positive eigenfunction
\begin{equation*}
\overline{e}_1(x,y)=\left\{(\lambda^{1/2}+1-\sigma)\,
\frac{\cosh\Big(y\sqrt{1-\lambda^{1/2}}\Big)}{\cosh\Big(\ell\sqrt{1-\lambda^{1/2}}\Big)}+
(\lambda^{1/2}-1+\sigma)\,
\frac{\cosh\Big(y\sqrt{1+\lambda^{1/2}}\Big)}{\cosh\Big(\ell\sqrt{1+\lambda^{1/2}}\Big)}\right\}\, \sin x\, .
\end{equation*}
\end{theorem}

The simplicity of the least eigenvalue was not to be expected. It is shown in \cite[\S3]{knight3} that the eigenvalue problem
\eq{linearproblem} for a fully hinged (simply supported) rectangular plate, that is with $u=\Delta u=0$ on the four edges, may admit a
least eigenvalue of multiplicity 2.\par The least eigenvalue $\lambda_1$ represents the critical buckling load and may be characterised variationally by
$$
\lambda_1\, :=\, \min_{v\in H^2_*(\Omega)}\ \frac{\|v\|_{H^2_*(\Omega)}^2}{\|v_x\|_{L^2(\Omega)}^2}\ .
$$
Ferrero-Gazzola \cite{fergaz} studied the eigenvalue problem $\Delta^2u=\lambda u$ under the boundary conditions in
\eq{linearproblem}: by comparing \cite[Theorem 3.4]{fergaz} with the above Theorem \ref{spectrum} we observe that the least eigenvalues
(and eigenfunctions) of the two problems coincide, that is,
\neweq{lambda1}
\lambda_1\, =\, \min_{v\in H^2_*(\Omega)}\
\frac{\|v\|_{H^2_*(\Omega)}^2}{\|v_x\|_{L^2(\Omega)}^2}\, =\,
\min_{v\in H^2_*(\Omega)}\
\frac{\|v\|_{H^2_*(\Omega)}^2}{\|v\|_{L^2(\Omega)}^2}\ .
\endeq
Therefore, the critical buckling load for a rectangular plate equals the eigenvalue relative to the first eigenmode of the plate. In
turn, the first eigenmode is also the first buckling deformation of the plate. {From} \eq{lambda1} we readily infer the Poincar\'e-type inequalities
\neweq{poincarry}
\lambda_1\|v_x\|_{L^2(\Omega)}^2\le\|v\|_{H^2_*(\Omega)}^2\ ,\quad\lambda_1\|v\|_{L^2(\Omega)}^2\le\|v\|_{H^2_*(\Omega)}^2\qquad\forall v\in H^2_*(\Omega)
\endeq
with strict inequality unless $v$ minimises the ratio in \eq{lambda1}, that is, $v$ is a real multiple of $\overline{e}_1$. Note also that by taking
$v(x,y)=\sin x$ one finds that $\lambda_1<1$.\par
Finally, let us mention that Theorem \ref{spectrum} may be complemented with the explicit form of all the eigenfunctions: they
are $\sin(mx)$ ($m\in\N$) multiplied by trigonometric or hyperbolic functions with respect to $y$: we refer again to Section \ref{prspect}.

Then we insert an external load $f$ and we study the existence and multiplicity of solutions of \eq{mostro2}.

\begin{theorem}\label{loadf}
For all $f\in L^2(\Omega)$ and $\lambda\ge0$ \eqref{mostro2} admits a solution $(u,\Phi)\in H_*^2(\Omega)\times H_{**}^2(\Omega)$.
Moreover:\\
(i) if $\lambda\leq \lambda_1$ and $f=0$, then \eqref{mostro2} only admits the trivial solution $(u,\Phi)=(0,0)$;\\
(ii) if $\lambda\in(\lambda_k, \lambda_{k+1}]$ for some $k\geq 1$ and $f=0$, then \eqref{mostro2} admits at least $k$ pairs of
nontrivial solutions;\\
(iii) if $\lambda<\lambda_1$ there exists $K>0$ such that if $\|f\|_{L^2(\Omega)}<K$ then \eqref{mostro2} admits a unique solution
$(u,\Phi)\in H_*^2(\Omega)\times H_{**}^2(\Omega)$;\\
(iv) if $\lambda>\lambda_1$ there exists $K>0$ such that if $\|f\|_{L^2(\Omega)}<K$ then \eqref{mostro2} admits at least three solutions.
\end{theorem}

Theorem \ref{loadf} gives both uniqueness and multiplicity results. Since the solutions are obtained as critical points of an action functional,
they describe the stable and unstable equilibria positions of the plate. When both the buckling load $\lambda$ and the external load $f$ are small
there is just one possible equilibrium position. If one of them is large then multiple equilibrium positions may exist. The uniqueness statement (iii)
has a fairly delicate proof: we will show that the corresponding action functional is ``locally convex'' in the region where the equilibria positions
are confined.\par
The last step is to study the nonlinear plate modeling the suspension bridge, that is, with the action of the hangers. We first define the constants
\neweq{alpha}
\alpha:=\int_\Omega\Upsilon(y)\overline{e}_1^2\, ,\qquad\overline{\lambda}:=(\alpha k+1)\lambda_1>\lambda_1\, ,
\endeq
where $\lambda_1$ denotes the least eigenvalue and $\overline{e}_1$ denotes here the positive least eigenfunction normalised in
$H_*^2(\Omega)$, see Theorem \ref{spectrum}. Then we have

\begin{theorem}\label{withangers}
For all $f\in L^2(\Omega)$, $\lambda\ge0$ and $k,\delta>0$ problem \eqref{finalmostro} admits a solution $(u,\Phi)\in H_*^2(\Omega)\times H_{**}^2(\Omega)$.
Moreover:\\
(i) if $\lambda<\lambda_1$ there exists $K>0$ such that if $\|f\|_{L^2(\Omega)}<K$ then \eqref{finalmostro} admits a unique solution
$(u,\Phi)\in H_*^2(\Omega)\times H_{**}^2(\Omega)$;\\
(ii) if $\lambda>\lambda_1$ and $f=0$ then \eqref{finalmostro} admits at least two solutions
$(u,\Phi)\in H_*^2(\Omega)\times H_{**}^2(\Omega)$ and one of them is trivial and unstable;\\
(iii) if $\overline{\lambda}<\lambda_2$ and $\overline{\lambda}<\lambda<\lambda_2$, there exists $K>0$ such that if $\|f\|_{L^2(\Omega)}<K$
then \eqref{finalmostro} admits at least three solutions $(u,\Phi)\in H_*^2(\Omega)\times H_{**}^2(\Omega)$, two being stable and one being unstable.
\end{theorem}

Also Theorem \ref{withangers} gives both uniqueness and multiplicity results. Item (ii) states that even in absence of an external load ($f=0$),
if the buckling load $\lambda$ is sufficiently large then there exists at least two equilibrium positions; we conjecture that if we further assume that
$\lambda<\overline{\lambda}$ then there exist no other solutions and that the equilibrium positions look like in Figure \ref{twobridges}.
\begin{figure}[ht]
\begin{center}
{\includegraphics[height=16mm, width=166mm]{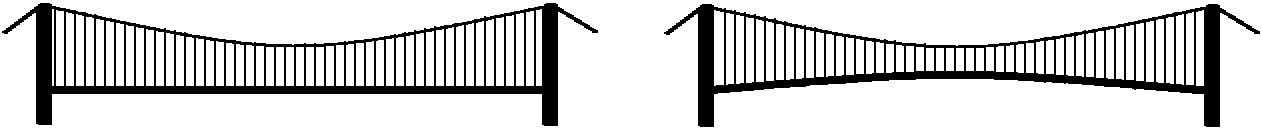}}
\caption{Equilibrium positions of the buckled bridge.}\label{twobridges}
\end{center}
\end{figure}
In the left picture we see the trivial equilibrium $u=0$ which is unstable due to the buckling load. In the right picture we see the stable equilibrium
for some $u<0$ (above the horizontal position). We conjecture that it is a negative multiple of the first eigenfunction $\overline{e}_1$, see
Theorem \ref{spectrum}; since $\ell$ is very small, a rough approximation shows that this negative multiple looks like $\approx C\sin(x)$ for some
$C<0$, which is the shape represented in the right picture. The reason of this conjecture will become clear in the proof, see in particular the plots in
Figure \ref{qualitative} in Section \ref{proofwith}: in this pattern, a crucial role is played by the positivity of $\overline{e}_1$. Our feeling is
that the action functional corresponding to this case has a qualitative shape as described in Figure \ref{saddlevon},
\begin{figure}[ht]
\begin{center}
{\includegraphics[height=40mm, width=42mm]{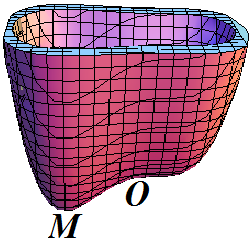}}
\caption{Qualitative shape of the action functional for Theorem \ref{withangers} (ii) when $\lambda<\overline{\lambda}$.}\label{saddlevon}
\end{center}
\end{figure}
where $O$ is the trivial unstable equilibrium and $M$ is the stable equilibrium. If there were no hangers also the opposite position would
be a stable equilibrium. But the presence of the restoring force requires a larger buckling term in order to generate a positive (downwards)
displacement. Indeed, item (iii) states, in particular, that if $f=0$ and the buckling load is large then there exist three equilibria: one
is trivial and unstable, the second is the enlarged negative one already found in item (ii), the third should precisely be the positive one
which appears because the buckling load $\lambda$ is stronger than the restoring force due to the hangers. All these conjectures and qualitative
explanations are supported by similar results for a simplified (one dimensional) beam equation, see \cite[Theorem 3.2]{giorgi3}.

\begin{remark}{\em {\bf (Open problem)}
Can the assumption $\overline{\lambda}<\lambda_2$ in Theorem \ref{withangers} $(iii)$ be weakened or removed? In our proof this assumption is needed to
disconnect two open regions of negativity of the action functional. But, perhaps, other critical point theorems may be applied.}
\end{remark}

\begin{remark}{\em {\bf (Regularity)} A weak solution satisfies $(u,\Phi)\in H_*^2(\Omega)\times H_{**}^2(\Omega)$: then the assumption $f\in L^2(\Omega)$ implies
that $\Delta^2 u\in L^1(\Omega)$. By an embedding and elliptic regularity we infer that $u\in H_*^2(\Omega)\cap H^{3-\varepsilon}(\Omega)$ for all $\varepsilon>0$
and then $D^2u\in H^{1-\varepsilon}(\Omega)$. Therefore, $[u,u]\in L^q(\Omega)$ for all $1\leq q<\infty$. Hence, $\Phi\in W^{4,q}(\Omega)$ and, in turn,
$[u,\Phi]\in L^q(\Omega)$ for all $1\leq q<\infty$. Moreover, $f\in L^2(\Omega)$ implies $\Delta^2 u\in L^2(\Omega)$ and then $u\in H^4(\Omega)$.
This means that the generalized solution $(u,\Phi)$ is also a strong solution.
For smoother $f$, the regularity of $(u,\Phi)$ can be increased.}
\end{remark}

\section{Preliminaries: some useful operators and functionals}

For all $v,w \in H_*^2(\Omega)$, consider the problem
\neweq{phiproblem}
\left\{\begin{array}{ll}
\Delta^2 \Phi=-[v,w]\quad & \mbox{in }\Omega\\
\Phi=\Phi_{xx}=0\quad & \mbox{on }\{0,\pi\}\times(-\ell,\ell)\\
\Phi=\Phi_y=0\quad & \mbox{on }(0,\pi)\times\{\pm\ell\}\, .
\end{array}\right.
\endeq
We claim that \eq{phiproblem} has a unique solution $\Phi=\Phi(v,w)$ and $\Phi\in H_{**}^2(\Omega)$.

Since $\Omega\subset \mathbb{R}^2$, we have $H^{1+\varepsilon}(\Omega)\Subset L^{\infty}(\Omega)=(L^1(\Omega))'$, for all
$\varepsilon >0$. On the other hand, $L^1(\Omega)\subset (L^{\infty}(\Omega))'\Subset H^{-(1+\varepsilon)}(\Omega)$.
If $v,w\in H_*^2(\Omega)\subset H^2(\Omega)$, then $[v,w]\in L^1(\Omega)$. Therefore,
$$[v,w]\in H^{-(1+\varepsilon)}(\Omega) \qquad \forall \varepsilon >0\, .$$
Then by the Lax-Milgram Theorem and the regularity theory of elliptic equations, there exists a
unique solution of \eq{phiproblem} and $\Phi\in H^{3-\varepsilon}(\Omega)$ for all $\varepsilon>0$.
An embedding and the boundary conditions show that $\Phi\in H_{**}^2(\Omega)$, which completes the proof of the claim.\par
This result enables us to define a bilinear form $B=B(v,w)=-\Phi$, where $\Phi$ is the unique solution of \eq{phiproblem}; this form is implicitly
characterised by
$$B:(H_*^2(\Omega))^2\rightarrow H_{**}^2(\Omega)\, ,\quad (B(v,w),\varphi)_{H_{**}^2(\Omega)}=\int_{\Omega}{[v,w]\varphi}\quad
\forall v,w\in H_*^2(\Omega)\, ,\ \varphi\in H_{**}^2(\Omega)\, .$$

Similarly, one can prove that for all $v\in H_*^2(\Omega)$ and $\varphi\in H_{**}^2(\Omega)$ there exists a unique solution
$\Psi\in H_*^2(\Omega)$ of the problem
$$
\left\{\begin{array}{ll}
\Delta^2 \Psi=-[v,\varphi]\quad & \mbox{in }\Omega\\
\Psi=\Psi_{xx}=0\quad & \mbox{on }\{0,\pi\}\times(-\ell,\ell)\\
\Psi_{yy}+\sigma\Psi_{xx}=\Psi_{yyy}+(2-\sigma)\Psi_{xxy}=0\quad & \mbox{on }(0,\pi)\times\{\pm\ell\}\, .
\end{array}\right.
$$
This defines another bilinear form $C=C(v,\varphi)=-\Psi$ which is implicitly characterised by
$$C:H_*^2(\Omega)\times H_{**}^2(\Omega)\rightarrow H_*^2(\Omega)\, ,\quad (C(v,\varphi),w)_{H_*^2(\Omega)}=\int_{\Omega}[v,\varphi]w\quad
\forall v,w\in H_*^2(\Omega)\, ,\ \varphi\in H_{**}^2(\Omega)\, .$$
Then we prove

\begin{lemma}\label{lemmaa}
The trilinear form
\neweq{tri}
(H_*^2(\Omega))^3\ni(v,w,\varphi)\mapsto\int_{\Omega}{[v,w]\varphi}
\endeq
is independent of the order of $v,w,\varphi$ if at least one of them is in $H_{**}^2(\Omega)$. Moreover,
if $\varphi\in H_{**}^2(\Omega), v,w\in (H_*^2(\Omega))^2$, then
\neweq{BC}
(B(v,w),\varphi)_{H_{**}^2(\Omega)}=(B(w,v),\varphi)_{H_{**}^2(\Omega)}=(C(v,\varphi),w)_{H_*^2(\Omega)}=(C(w,\varphi),v)_{H_*^2(\Omega)}.
\endeq
Finally, the operators $B$ and $C$ are compact.
\end{lemma}
\begin{proof} By a density argument and by continuity it suffices to prove all the identities for smooth functions $v,w,\varphi$, in such a way that
third interior derivatives and second boundary derivatives are well defined and integration by parts is allowed.
In the trilinear form \eq{tri} one can exchange the order of $v$ and $w$ by exploiting the symmetry of the Monge-Amp\`ere operator,
that is, $[v,w]=[w,v]$ for all $v$ and $w$. So, we may assume that one among $w,\varphi$ is in $H_{**}^2(\Omega)$: note that this function also has
vanishing $x$-derivative on $(0,\pi)\times\{\pm\ell\}$. Then some integration by parts enable to switch the position of $w$ and $\varphi$.\par
From the just proved symmetry of the trilinear form \eq{tri} we immediately infer \eq{BC}.\par
If $\varphi\in H_{**}^2(\Omega)$, then $\varphi_{xx}=\varphi_{xy}=0$ on $(0,\pi)\times\{\pm\ell\}$ and
an integration by parts yields
$$
(B(v,w),\varphi)_{H_{**}^2(\Omega)}=\int_{\Omega}{[v,w]\varphi}=\int_{\Omega}{[\varphi,w]v}
=\int_{\Omega}{\varphi_{xy}(w_x v_y+w_y v_x)}-\int_{\Omega}{(\varphi_{xx}w_y v_y+\varphi_{yy}w_x v_x)}.
$$
In turn, this shows that
$$
|(B(v,w),\varphi)_{H_{**}^2(\Omega)}|\leq c\|\varphi\|_{H_{**}^2(\Omega)}\|v\|_{W^{1,4}(\Omega)}\|w\|_{W^{1,4}(\Omega)}, \qquad
\forall v,w \in H_*^2(\Omega), \forall \varphi \in H_{**}^2(\Omega).
$$
Therefore,
\neweq{boundB}
\|B(v,w)\|_{H_{**}^2(\Omega)}=\sup_{0\neq \varphi\in H_{**}^2(\Omega)}\frac{(B(v,w),\varphi)_{H_{**}^2(\Omega)}}
{\|\varphi\|_{H_{**}^2(\Omega)}}\leq c\|v\|_{W^{1,4}(\Omega)}\|w\|_{W^{1,4}(\Omega)}.
\endeq
Assume that the sequence $\{(v_n,w_n)\}\subset H_*^2(\Omega)$ weakly converges to $(v,w)\in H_*^2(\Omega)$. Then the triangle inequality and the
just proved estimate yield
\begin{eqnarray*}
\|B(v_n,w_n)-B(v,w)\|_{H_{**}^2(\Omega)} &\leq& \|B(v_n-v,w_n)\|_{H_{**}^2(\Omega)}+\|B(v,w_n-w)\|_{H_{**}^2(\Omega)}\\
\ &\leq& c\|v_n-v\|_{W^{1,4}(\Omega)}\|w_n\|_{W^{1,4}(\Omega)}+c\|v\|_{W^{1,4}(\Omega)}\|w_n-w\|_{W^{1,4}(\Omega)}\, .
\end{eqnarray*}
The compact embedding $H_*^2(\Omega)\Subset W^{1,4}(\Omega)$ then shows that
$$\|B(v_n,w_n)-B(v,w)\|_{H_{**}^2(\Omega)} \rightarrow 0$$
and hence that $B$ is a compact operator. The proof for $C$ is similar.\end{proof}

We now define another operator $D:H_*^2(\Omega)\rightarrow H_*^2(\Omega)$ by
$$D(v)=C(v,B(v,v)) \qquad \forall v\in H_*^2(\Omega)$$
and we prove

\begin{lemma} \label{lemmac}
The operator $D$ is compact.
\end{lemma}
\begin{proof}
Assume that the sequence $\{v_n\}\subset H_*^2(\Omega)$ weakly converges to $v \in H_*^2(\Omega)$. Then, by Lemma \ref{lemmaa},
$$B(v_n,v_n)\rightarrow B(v,v) \mbox{ in } H_{**}^2(\Omega),\qquad C(v_n, B(v_n,v_n))\rightarrow C(v, B(v,v)) \mbox{ in }H_*^2(\Omega).$$
This proves that $D(v_n)\rightarrow D(v)$ in $H_*^2(\Omega)$ and that $D$ is a compact operator.
\end{proof}

In turn, the operator $D$ enables us to define a functional $d: H_*^2(\Omega)\rightarrow \mathbb{R}$ by
$$d(v)=\frac{1}{4}(D(v),v)_{H_*^2(\Omega)} \qquad \forall v \in H_*^2(\Omega)\, .$$
In the next statement we prove some of its properties.

\begin{lemma} \label{lemmad}
The functional $d: H_*^2(\Omega)\rightarrow \mathbb{R}$ has the following properties:\\
(i) $d$ is nonnegative and $d(v)=0$ if and only if $v=0$ in $\Omega$.
Moreover,
$$d(v)=\frac{1}{4}\|B(v,v)\|_{H_{**}^2(\Omega)}^2;$$
(ii) $d$ is quartic, i.e., $$d(r v)=r^4 d(v),\quad \forall r\in\mathbb{R}, \forall v\in H_*^2(\Omega);$$
(iii) $d$ is differentiable in $H_*^2(\Omega)$ and
$$\langle d'(v),w\rangle=(D(v),w)_{H_*^2(\Omega)}, \qquad v,w \in H_*^2(\Omega);$$
(iv) $d$ is weakly continuous on $H_*^2(\Omega)$.
\end{lemma}
\begin{proof} (i) By \eq{BC} we know that for any $v\in H_*^2(\Omega)$,
$$(D(v),v)_{H_*^2(\Omega)}=(C(v,B(v,v)),v)_{H_*^2(\Omega)}=(B(v,v),B(v,v))_{H_{**}^2(\Omega)}=\|B(v,v)\|_{H_{**}^2(\Omega)}^2.$$
Whence, if $d(v)=0$, then $B(v,v)=0$ and $[v,v]=0$, see \eq{phiproblem}. But $[v,v]$ is proportional to the Gaussian curvature and since it
vanishes identically this implies that the surface $v=v(x,y)$ is covered by straight lines. By using the boundary condition \eq{bcuno} we finally
infer that $v\equiv0$. This idea of the last part of this proof is taken from \cite[Lemma 3.2']{bergerfife2}.\par
(ii) The functional $d$ is quartic as a trivial consequence of its definition.\par
(iii) From \eq{BC} we infer that
\neweq{BC1}
(C(v,B(v,w)),v)_{H_*^2(\Omega)}=(B(v,v),B(v,w))_{H_{**}^2(\Omega)}=(C(v,B(v,v)),w)_{H_*^2(\Omega)}\quad\forall v,w\in H_*^2(\Omega).
\endeq
Then we compute
\begin{align*}
\langle d'(v),w\rangle&=\lim_{\varepsilon\to0}\frac{1}{4\varepsilon}\{(D(v+\varepsilon w),v+\varepsilon w)_{H_*^2(\Omega)}-(D(v),v)_{H_*^2(\Omega)}\}\\
&=\lim_{\varepsilon\to0}\frac{1}{4\varepsilon}\{(C(v+\varepsilon w, B(v+\varepsilon w,v+\varepsilon w)),
v+\varepsilon w)_{H_*^2(\Omega)}-(C(v,B(v,v)),v)_{H_*^2(\Omega)}\}\\
&=\frac{1}{4}\{(C(w, B(v,v)),v)_{H_*^2(\Omega)}+(C(v, B(v,v)),w)_{H_*^2(\Omega)}+2(C(v,B(v,w)),v)_{H_*^2(\Omega)}\}\\
\mbox{by \eq{BC} }&=\frac{1}{2}\{(C(v, B(v,v)),w)_{H_*^2(\Omega)}+(C(v,B(v,w)),v)_{H_*^2(\Omega)}\}\\
\mbox{by \eq{BC1} }&=(D(v),w)_{H_*^2(\Omega)},
\end{align*}
which proves (iii).\par
(iv) Assume that the sequence $\{v_n\}\subset H_*^2(\Omega)$ weakly converges to $v \in H_*^2(\Omega)$. Then by Lemma \ref{lemmac} we know that
$$\lim_{n\rightarrow\infty}\|D(v_n)-D(v)\|_{H_*^2(\Omega)}=0.$$
This shows that
$$\lim_{n\rightarrow\infty}(D(v_n)-D(v),v_n)_{H_*^2(\Omega)}=0.$$
Finally, this yields
$$d(v_n)-d(v)=\frac{1}{4}(D(v_n)-D(v),v_n)_{H_*^2(\Omega)}+\frac{1}{4}(D(v),v_n-v)_{H_*^2(\Omega)}\rightarrow0$$
which proves (iv).\end{proof}

\section{Proof of Theorem \ref{spectrum}}\label{prspect}

In this section we prove Theorem \ref{spectrum} and we give some
more details about the eigenvalues and eigenfunctions of
\eq{linearproblem}. We proceed as in \cite[Theorem 3.4]{fergaz}, see
also \cite[Theorem 4]{algwaiz}, with some changes due to the
presence of the buckling term. We write the eigenvalue problem
\eqref{linearproblem} as
$$(u_x,v_x)_{L^2(\Omega)}=\frac{1}{\lambda}(u,v)_{H^2_*(\Omega)}\qquad\forall v\in H^2_*(\Omega).$$
Define the linear operator $T:H^2_*(\Omega)\to H^2_*(\Omega)$ such
that
$$(Tu,v)_{H^2_*(\Omega)}=(u_x,v_x)_{L^2(\Omega)}\qquad\forall v\in H^2_*(\Omega).$$
The operator $T$ is self-adjoint since
$$(Tu,v)_{H^2_*(\Omega)}=(u_x,v_x)_{L^2(\Omega)}=(v_x,u_x)_{L^2(\Omega)}=(u,Tv)_{H^2_*(\Omega)}\qquad\forall u,v\in H^2_*(\Omega)\, .$$
Moreover, by the compact embedding $H^2_*(\Omega)\Subset
H^1(\Omega)$ and the definition of $T$, the following implications
hold:
\begin{eqnarray*}
u_n\rightharpoonup u\mbox{ in }H^2_*(\Omega) & \Longrightarrow &
(u_n)_x\to u_x\mbox{ in }L^2(\Omega)\ \Longrightarrow\
\sup_{\|v\|_{H^2_*(\Omega)}=1}\ ((u_n-u)_x,v_x)_{L^2(\Omega)}\to0\\
\ & \Longrightarrow & \sup_{\|v\|_{H^2_*(\Omega)}=1}\
(T(u_n-u),v)_{H^2_*(\Omega)}\to0\ \Longrightarrow\ Tu_n\to Tu\mbox{
in }H^2_*(\Omega)
\end{eqnarray*}
which shows that $T$ is also compact. Then the spectral theory of
linear compact self-adjoint operator yields that
\eqref{linearproblem} admits an ordered increasing sequence of
eigenvalues and the corresponding eigenfunctions form an Hilbertian
basis of $H^2_*(\Omega)$. This proves the first part of Theorem
\ref{spectrum}.\par According to the boundary conditions on
$x=0,\pi$, we seek eigenfunctions in the form:
\neweq{u}
u(x,y)=\sum_{m=1}^{+\infty} h_m(y)\sin(mx) \qquad \text{for
}(x,y)\in (0,\pi)\times(-\ell,\ell)\, .
\endeq
Then we are led to find nontrivial solutions of the ordinary
differential equation
\begin{equation}\label{eq:ODE-h_m}
h_m''''(y)-2m^2 h_m''(y)+(m^4-m^2\lambda)h_m(y)=0\,
,\qquad(\lambda>0)
\end{equation}
with the boundary conditions
\begin{equation} \label{eq:bound-cond}
h_m''(\pm \ell)-\sigma m^2h_m(\pm\ell)=0 \, , \qquad
h_m'''(\pm\ell)+(\sigma-2)m^2 h_m'(\pm\ell)=0 \, .
\end{equation}
The characteristic equation related to \eqref{eq:ODE-h_m} is
$\alpha^4-2m^2\alpha^2+m^4-m^2\lambda=0$ and then
\begin{equation} \label{eq:alpha^2}
\alpha^2=m^2\pm m\sqrt{\lambda}\, .
\end{equation}
For a given $\lambda>0$ three cases have to be distinguished.\par
$\bullet$ {\bf The case $m^2>\lambda$.} By \eqref{eq:alpha^2} we
infer
\neweq{betagamma}
\alpha=\pm \beta\mbox{ or  }\alpha=\pm
\gamma\qquad\mbox{with}\qquad\sqrt{m^2-m\sqrt{\lambda}}=:\gamma<\beta:=\sqrt{m^2+m\sqrt{\lambda}}\,
.
\endeq
Nontrivial solutions of \eqref{eq:ODE-h_m} have the form
\begin{equation} \label{eq:prima-hm}
h_m(y)=a\cosh(\beta y)+b\sinh(\beta y)+c\cosh(\gamma
y)+d\sinh(\gamma y)\qquad(a,b,c,d\in\R) \, .
\end{equation}
By imposing the boundary conditions \eqref{eq:bound-cond} and
arguing as in \cite{fergaz} we see that a nontrivial solution of
\eq{eq:ODE-h_m} exists if and only if one of the two following
equalities holds:
\neweq{primaeq}
\frac{\gamma}{(\gamma^2-m^2\sigma)^2}\,
\tanh(\ell\gamma)=\frac{\beta}{(\beta^2-m^2\sigma)^2}\,
\tanh(\ell\beta)\, ,
\endeq
\neweq{secondaeq}
\frac{\beta}{(\beta^2-m^2\sigma)^2}\,
\coth(\ell\beta)=\frac{\gamma}{(\gamma^2-m^2\sigma)^2}\,
\coth(\ell\gamma)\, .
\endeq

For any integer $m>\sqrt{\lambda}$ such that \eq{primaeq} holds, the
function $h_m$ in \eq{eq:prima-hm} with $b=d=0$ and suitable
$a=a_m\neq0$ and $c=c_m\neq0$ yields the eigenfunction
$h_m(y)\sin(mx)$ associated to the eigenvalue $\lambda$. Similarly,
for any integer $m>\sqrt{\lambda}$ such that \eq{secondaeq} holds,
the function $h_m$ in \eq{eq:prima-hm} with $a=c=0$ and suitable
$b=b_m\neq0$ and $d=d_m\neq0$ yields the eigenfunction
$h_m(y)\sin(mx)$ associated to the eigenvalue $\lambda$. Clearly,
the number of both such integers is finite. In particular, when
$m=1$ the equation \eq{eq:ODE-h_m} coincides with
\cite[(57)]{fergaz}. Therefore, the statement about the least
eigenvalue and the explicit form of the corresponding eigenfunction
hold.\par $\bullet$ {\bf The case $m^2=\lambda$.} This case is
completely similar to the second case in \cite{fergaz}. By
\eqref{eq:alpha^2} we infer that possible nontrivial solutions of
\eqref{eq:ODE-h_m}-\eq{eq:bound-cond} have the form
$$
h_m(y)=a\cosh(\sqrt{2}my)+b\sinh(\sqrt{2}my)+c+dy\qquad(a,b,c,d\in\R)\, .
$$
Then one sees that $a=c=0$ if \eqref{sigma} holds. Moreover, let
$\overline{s}>0$ the unique solution of
$\tanh(s)=\left(\frac{\sigma}{2-\sigma}\right)^2s$. If
$m_*:=\overline{s}/\ell\sqrt2$ is an integer, and only in this case,
then $\lambda=m_*^2$ is an eigenvalue and the corresponding eigenfunction is
$$\Big[\sigma\ell\sinh(\sqrt2 m_*y)+(2-\sigma)\sinh(\sqrt2 m_*\ell)\, y\Big]\, \sin(m_*x)\, .$$

$\bullet$ {\bf The case $m^2<\lambda$.} By \eqref{eq:alpha^2} we infer that
$$
\alpha=\pm \beta\mbox{ or }\alpha=\pm i\gamma\mbox{ \ with \
}\sqrt{m\sqrt{\lambda}-m^2}=\gamma<\beta=\sqrt{m\sqrt{\lambda}+m^2}\, .
$$
Therefore, possible nontrivial solutions of \eq{eq:ODE-h_m} have the form
$$
h_m(y)=a\cosh(\beta y)+b\sinh(\beta y)+c\cos(\gamma y)+d\sin(\gamma
y)\qquad(a,b,c,d\in\R)\, .
$$
Differentiating $h_m$ and imposing the boundary conditions \eqref{eq:bound-cond} yields the two systems:
$$
\left\{\begin{array}{ll}
(\beta^2-m^2\sigma)\cosh(\beta\ell)a-(\gamma^2+m^2\sigma)\cos(\gamma\ell)c=0 \\
(\beta^3-m^2(2-\sigma)\beta)\sinh(\beta\ell)a+(\gamma^3+m^2(2-\sigma)\gamma)\sin(\gamma\ell)c=0\, ,
\end{array}\right.
$$
$$
\left\{\begin{array}{ll}
(\beta^2-m^2\sigma)\sinh(\beta\ell)b-(\gamma^2+m^2\sigma)\sin(\gamma\ell)d=0\\
(\beta^3-m^2(2-\sigma)\beta)\cosh(\beta\ell)b-(\gamma^3+m^2(2-\sigma)\gamma)\cos(\gamma\ell)d=0\, .
\end{array}\right.
$$
Due to the presence of trigonometric sine and cosine, for any integer $m$ there exists a sequence $\zeta_k^m\uparrow +\infty$ such
that $\zeta_k^m>m^2$ for all $k\in\N$ and such that if $\lambda=\zeta_k^m$ for some $k$ then one of the above systems
admits a nontrivial solution. On the other hand, for any eigenvalue $\lambda$ there exists at most a finite number of integers $m$ such
that $m^2<\lambda$; if these integers yield nontrivial solutions $h_m$, then the function $h_m(y)\sin(mx)$ is an eigenfunction
corresponding to $\lambda$.

\section{Proof of Theorem \ref{loadf}}

By Lemma \ref{lemmad} we know that a functional whose critical points are solutions of the problem \eq{mostro2} reads
$$J(u)=\frac{1}{2}\|u\|_{H_*^2(\Omega)}^2+d(u)-\frac{\lambda}{2}\|u_x\|_{L^2(\Omega)}^2-\int_\Omega{fu}\qquad\forall u\in H_*^2(\Omega).$$
By combining Lemmas \ref{lemmaa}-\ref{lemmac}-\ref{lemmad}, we obtain a one-to-one correspondence between solutions of \eq{mostro2}
and critical points of the functional $J$:

\begin{lemma}\label{critical}
Let $f\in L^2(\Omega)$. The couple $(u,\Phi)\in H^2_*(\Omega)\times H^2_{**}(\Omega)$ is a weak solution of \eqref{mostro2} if and only if
$u\in H^2_*(\Omega)$ is a critical point of $J$ and if $\Phi\in H^2_{**}(\Omega)$ weakly solves $\Delta^2\Phi=-[u,u]$ in $\Omega$.
\end{lemma}

The first step is then to prove geometrical properties (coercivity) and compactness properties (Palais-Smale condition) of $J$. Although the
former may appear straightforward, it requires delicate arguments. The reason is that no useful lower bound for $d(u)$ is available. We prove

\begin{lemma} \label{coercive}
For any $f\in L^2(\Omega)$ and any $\lambda\ge0$, the functional $J$ is coercive in $H_*^2(\Omega)$ and it is bounded from below. Moreover,
it satisfies the Palais-Smale (PS) condition.
\end{lemma}
\begin{proof} Assume for contradiction that there exists a sequence $\{v_n\}\subset H_*^2(\Omega)$ and $M>0$ such that
$$\lim_{n\to \infty}\|v_n\|_{H_*^2(\Omega)}\to \infty, \qquad J(v_n)\leq M.$$
Put $w_n=\frac{v_n}{\|v_n\|_{H_*^2(\Omega)}}$ so that $v_n=\|v_n\|_{H_*^2(\Omega)}w_n$ and
\neweq{norm1}
\|w_n\|_{H_*^2(\Omega)}=1\qquad\forall n\, .
\endeq
By combining the H\"older inequality with \eq{poincarry}, we infer that
\neweq{lowerb}
M\ge J(v_n)\ge \frac{1}{2}\|v_n\|_{H_*^2(\Omega)}^2+\|v_n\|_{H_*^2(\Omega)}^4d(w_n)-\frac{\lambda}{2}\|v_n\|_{H_*^2(\Omega)}^2\|(w_n)_x\|_{L^2(\Omega)}^2
-\frac{\|f\|_{L^2(\Omega)}}{\sqrt{\lambda_1}}\|v_n\|_{H_*^2(\Omega)},
\endeq
where we also used Lemma \ref{lemmad} (ii). By letting $n\to\infty$, this shows that $d(w_n)\to0$ which, combined with Lemma \ref{lemmad} and \eq{norm1},
shows that $w_n\rightharpoonup0$ in $H_*^2(\Omega)$; then, $(w_n)_x\rightarrow 0$ in $L^2(\Omega)$ by compact embedding. Hence, since $d(w_n)\ge 0$,
\eq{lowerb} yields
$$
o(1)=\frac{M}{\|v_n\|_{H_*^2(\Omega)}^2}\ge \frac{1}{2}+\|v_n\|_{H_*^2(\Omega)}^2d(w_n)
-\frac{\lambda}{2}\|(w_n)_x\|_{L^2(\Omega)}^2-\frac{\|f\|_{L^2(\Omega)}}{\|v_n\|_{H_*^2(\Omega)}\sqrt{\lambda_1}}\ge \frac{1}{2}+o(1)
$$
which leads to a contradiction by letting $n\to\infty$. Therefore $J$ is coercive. Since the lower bound for $J(v_n)$ in \eq{lowerb}
only depends on $\|v_n\|_{H_*^2(\Omega)}$, we also know that $J$ is bounded from below.\par
In order to prove that $J$ satisfies the (PS) condition we consider a sequence $\{u_n\}\subset H_*^2(\Omega)$ such that $J(u_n)$ is
bounded and $J'(u_n)\to0$ in $\HH_*(\Omega)$. By what we just proved, we know that $\{u_n\}$ is bounded and therefore, there exists
$\overline{u}\in H_*^2(\Omega)$ such that $u_n\rightharpoonup\overline{u}$ and, by weak continuity, $J'(\overline{u})=0$. Moreover, by Lemma \ref{lemmad},
\begin{align*}
&\langle J'(u_n),u_n\rangle=\|u_n\|_{H_*^2(\Omega)}^2+(D(u_n),u_n)_{H_*^2(\Omega)}-\lambda\|(u_n)_x\|_{L^2(\Omega)}^2-\int_\Omega{fu_n}\to\\
&\to 0 =\langle J'(\overline{u}),\overline{u}\rangle=\|\overline{u}\|_{H_*^2(\Omega)}^2+(D(\overline{u}),\overline{u})_{H_*^2(\Omega)}
-\lambda\|\overline{u}_x\|_{L^2(\Omega)}^2-\int_\Omega{f\overline{u}}\, .\end{align*}
Since $(D(u_n),u_n)_{H_*^2(\Omega)}\to(D(\overline{u}),\overline{u})_{H_*^2(\Omega)}$ by Lemma \ref{lemmac}, $\|(u_n)_x\|_{L^2(\Omega)}^2\to\|\overline{u}_x\|_{L^2(\Omega)}^2$
and $\int_\Omega{fu_n}\to\int_\Omega{f\overline{u}}$ by compact embedding, this proves that $\|u_n\|_{H_*^2(\Omega)}\to\|\overline{u}\|_{H_*^2(\Omega)}$.
This fact, together with the weak convergence $u_n\rightharpoonup\overline{u}$ proves that, in fact, $u_n\to\overline{u}$ strongly; this proves (PS).\end{proof}

Lemma \ref{coercive} shows that the (smooth) functional $J$ admits a global minimum in $H_*^2(\Omega)$ for any $f$ and $\lambda$. This minimum is
a critical point for $J$ and hence, by Lemma \ref{critical}, it gives a weak solution of \eqref{mostro2}. This proves the first part of Theorem \ref{loadf}.
Let us now prove the items.\par
(i) If $\lambda\leq \lambda_1$ and $f=0$, we see that any critical point $u$ of $J$ satisfies
$$0=\langle J'(u),u\rangle=\|u\|_{H_*^2(\Omega)}^2+4d(u)-\lambda\|u_x\|_{L^2(\Omega)}^2$$
where we also used Lemma \ref{lemmad} (iii). By Lemma \ref{lemmad} i) and \eq{poincarry}, this proves that $u=0$. Then we apply again Lemma
\ref{critical} and find $(u,\Phi)=(0,0)$.\par
(ii) If $f=0$ and $\lambda\in(\lambda_k,\lambda_{k+1}]$, then the twice differentiable functional $J$ is even and its second derivative $J''(0)$ at $0$
has Morse index $k$. By Lemma \ref{coercive} we may then apply \cite[Theorem 11]{algwaiz} (which is a variant of Theorem 5.2.23 p.369 in \cite{chang}),
to infer that $J$ has at least $k$ pairs of district nonzero critical points. Then by Lemma \ref{critical} there exist at least $k$ pairs of nontrivial
solutions of \eq{mostro2}.\par
(iii) For any $f\in L^2(\Omega)$, if $u$ is a critical point of the functional $J$ it satisfies $\langle J'(u),u\rangle=0$ and therefore,
by the H\"older inequality,
$$\|u\|_{H_*^2(\Omega)}^2+4d(u)-\lambda\|u_x\|_{L^2(\Omega)}^2\le\|f\|_{L^2(\Omega)}\|u\|_{L^2(\Omega)}\, .$$
In turn, by using Lemma \ref{lemmad} i) and twice \eq{poincarry}, we obtain
$$\left(1-\frac{\lambda}{\lambda_1}\right)\|u\|_{H_*^2(\Omega)}^2\le\frac{\|f\|_{L^2(\Omega)}}{\sqrt{\lambda_1}}\|u\|_{H^2_*(\Omega)}\, .$$
This gives the a priori bound
\neweq{apriori}
\|u\|_{H_*^2(\Omega)}\le\frac{\sqrt{\lambda_1}}{\lambda_1-\lambda}\, \|f\|_{L^2(\Omega)}\, .
\endeq

Next, we prove a local convexity property of the functional $J$. Let
$$Q(u):=\|u\|_{H_*^2(\Omega)}^2-\lambda\|u_x\|_{L^2(\Omega)}^2\quad\forall u\in H_*^2(\Omega)\, .$$
Then, for all $u,v\in H_*^2(\Omega)$ and all $t\in[0,1]$, we have
\neweq{QQ}
Q\Big(tu+(1-t)v\Big)-tQ(u)-(1-t)Q(v)=-t(1-t)\Big(\|u-v\|_{H_*^2(\Omega)}^2-\lambda\|u_x-v_x\|_{L^2(\Omega)}^2\Big)\, .
\endeq
Moreover, for all $u,v\in H_*^2(\Omega)$ and all $t\in[0,1]$, some tedious computations show that
$$
d\Big(tu+(1-t)v\Big)-td(u)-(1-t)d(v) =
$$
$$
= -\frac{t(1-t)}{4}\Big\{(t^2-3t+1)(\|B(v,u-v)\|_{H_{**}^2(\Omega)}^2-\|B(u,u-v)\|_{H_{**}^2(\Omega)}^2)$$
$$+2\big(B(v,v),B(v-u,v-u)\big)_{H_{**}^2(\Omega)}+2(t^2-t+1)\big(B(u,u-v),B(u+v,u-v)\big)_{H_{**}^2(\Omega)}$$
$$-4t(1-t)\big(B(u-v,u),B(v-u,v)\big)_{H_{**}^2(\Omega)}\Big\}$$
\neweq{dd}
\mbox{by \eq{boundB} }\ \le C\, t(1-t)\, (\|u\|_{H_*^2(\Omega)}^2+\|v\|_{H_*^2(\Omega)}^2)\, \|u-v\|_{H_*^2(\Omega)}^2\, ;
\endeq
here $C>0$ is a constant independent of $t$, $u$, $v$. Consider the ``unforced'' functional
\neweq{J0}
J_0(u)=\frac{1}{2}\|u\|_{H_*^2(\Omega)}^2+d(u)-\frac{\lambda}{2}\|u_x\|_{L^2(\Omega)}^2=\frac{Q(u)}{2}+d(u)\, ;
\endeq
by putting together \eq{QQ} and \eq{dd} we see that
$$
J_0\Big(tu+(1-t)v\Big)-tJ_0(u)-(1-t)J_0(v)\le
$$
$$
\le-\frac{t(1-t)}{2}\Big(\|u-v\|_{H_*^2(\Omega)}^2-\lambda\|u_x-v_x\|_{L^2(\Omega)}^2\Big)+
C\, t(1-t)\, (\|u\|_{H_*^2(\Omega)}^2+\|v\|_{H_*^2(\Omega)}^2)\, \|u-v\|_{H_*^2(\Omega)}^2
$$
\neweq{tecnici}
\le t(1-t)\Big(C\, (\|u\|_{H_*^2(\Omega)}^2+\|v\|_{H_*^2(\Omega)}^2)-\frac{\lambda_1-\lambda}{2\lambda_1}\Big)\|u-v\|_{H_*^2(\Omega)}^2\, .
\endeq
Take $f$ sufficiently small such that
\neweq{smallf}
\|f\|_{L^2(\Omega)}^2<K^2:=\frac{(\lambda_1-\lambda)^3}{4C\, \lambda_1^2}\, .
\endeq
By \eq{apriori} and \eq{smallf} we know that any critical point of $J$ satisfies
$$
\|u\|_{H_*^2(\Omega)}^2\le\frac{\lambda_1}{(\lambda_1-\lambda)^2}\, K^2=\frac{\lambda_1-\lambda}{4C\, \lambda_1}=:\rho^2\, ;
$$
put $B_\rho=\{u\in H_*^2(\Omega);\, \|u\|_{H_*^2(\Omega)}\le\rho\}$. Moreover, from \eq{tecnici} we know that
$$J_0\Big(tu+(1-t)v\Big)-tJ_0(u)-(1-t)J_0(v)\le0\qquad \forall u,v\in B_\rho\ ,$$
with strict inequality if $u\neq v$ and $t\not\in\{0,1\}$. This proves that $J_0$ is strictly convex in $B_\rho$ and since $J(u)$ equals
$J_0(u)$ plus a linear term (with respect to $u$), also $J$ is strictly convex in $B_\rho$.\par
Summarising, if \eq{smallf} holds, then we know that:\\
$\bullet$ by \eq{apriori} all the critical points of $J$ belong to $B_\rho$;\\
$\bullet$ by the first part of the proof we then know that there exists at least a critical point in $B_\rho$;\\
$\bullet$ $J$ is strictly convex in $B_\rho$.\par
We then deduce that $J$ admits a unique critical point in $B_\rho$ (its absolute minimum) and no other critical points elsewhere.
Together with Lemma \ref{critical}, this completes the proof of item (iii).\par
(iv) If $\lambda>\lambda_1$ we know from item (ii) that the unforced functional $J_0$ defined in \eq{J0} has two global minima $\pm\bar{u}\neq 0$.
Then a sufficiently small linear perturbation of $J_0$ has a local minimum in a neighborhood of both $\pm\bar{u}$. Whence, if $f$ is sufficiently
small, say $\|f\|_{L^2(\Omega)}<K$, then the functional $J$ defined by $J(u)=J_0(u)-\integ fu$ admits two local minima in two neighborhoods of both
$\pm\bar{u}$. These local minima, which we name $u_1$ and $u_2$, are the first two critical points of $J$. A minimax procedure then yields an additional
(mountain-pass) solution. Indeed, consider the set of continuous paths connecting $u_1$ and $u_2$:
$$
\Gamma:=\Big\{p\in C^0([0,1],H^2_*(\Omega));\, p(0)=u_1,\, p(1)=u_2\Big\}\, .
$$
Since by Lemma \ref{coercive} the functional $J$ satisfies the (PS) condition, the mountain-pass Theorem guarantees that the level
$$
\min_{p\in\Gamma}\ \max_{t\in[0,1]}\ J\big(p(t)\big)\, >\, \max\Big\{J(u_1),J(u_2)\Big\}
$$
is a critical level for $J$; this yields a third critical point. By Lemma \ref{critical} this proves the existence of (at least) three weak
solutions of \eq{mostro2}.

\section{Proof of Theorem \ref{withangers}}\label{proofwith}

Similar to Lemma \ref{critical}, the functional whose critical points are solutions of problem \eq{finalmostro} is
$$
J(u)=\frac{1}{2}\|u\|_{H_*^2(\Omega)}^2+\int_\Omega{\Upsilon(y)\left(\frac{k}{2}(u^+)^2+\frac{\delta}{4}
(u^+)^4\right)}+d(u)-\frac{\lambda}{2}\|u_x\|_{L^2(\Omega)}^2-\int_\Omega{fu}\, .
$$
And similar to Lemma \ref{coercive} one can prove that for any $f\in L^2(\Omega)$ and any $\lambda\ge0$, the functional $J$ is coercive in $H_*^2(\Omega)$,
it is bounded from below and it satisfies the (PS) condition. Then the smooth functional $J$ admits a global minimum in $H_*^2(\Omega)$ for any $f$ and
$\lambda$. This minimum is a critical point for $J$ and hence a weak solution of \eqref{finalmostro}. This proves the first part of Theorem \ref{withangers}.
Let us now prove the items.\par
(i) The proof of this item follows the same steps as item (iii) of Theorem \ref{loadf}: it suffices to notice that the additional term
$\int_\Omega{\Upsilon(y)\left(\frac{k}{2}(u^+)^2+\frac{\delta}{4}(u^+)^4\right)}$ is also convex.\par
(ii) If $f=0$, then $u=0$ is a solution for any $\lambda\ge0$. We just need to show that it is not the global minimum which we know to exist.
Let $\overline{e}_1$ and $\alpha$ be as in \eq{alpha} and consider the function
\neweq{gf}
g(t):=J(t\overline{e}_1)=-\frac{\lambda-\lambda_1}{2\lambda_1}\, t^2+\frac{k\, \alpha}{2}\, (t^+)^2
+\frac{\delta\, (t^+)^4}{4}\integ\Upsilon(y)\overline{e}_1^4+t^4\, d(\overline{e}_1)\qquad t\in\R\, .
\endeq
Since $\lambda>\lambda_1$, the coefficient of $(t^-)^2$ is negative and the qualitative graph of $g$ is as in Figure \ref{qualitative}
(on the left the case where $\lambda<\overline{\lambda}$ so that the coefficient of $(t^+)^2$ is nonnegative, on the right the case where also the coefficient of
$(t^+)^2$ is negative).
\begin{figure}[ht]
\begin{center}
{\includegraphics[height=16mm, width=30mm]{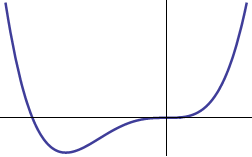}}\qquad\qquad{\includegraphics[height=16mm, width=30mm]{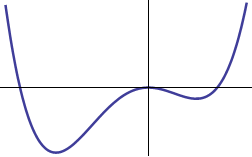}}
\caption{Qualitative graphs of the functions $g$ (left) and $h$ (right).}\label{qualitative}
\end{center}
\end{figure}
It is clear that there exists $\overline{t}<0$ such that $g(\overline{t})<0$. This means that $J(\overline{t}\overline{e}_1)<0$ and that $0$ is not
the absolute minimum of $J$. This completes the proof of item (ii).\par
(iii) We study first the case where $f=0$ and we name $J_0$ the unforced functional, that is,
$$
J_0(u)=\frac{1}{2}\|u\|_{H_*^2(\Omega)}^2+\int_\Omega{\Upsilon(y)\left(\frac{k}{2}(u^+)^2+\frac{\delta}{4}
(u^+)^4\right)}+d(u)-\frac{\lambda}{2}\|u_x\|_{L^2(\Omega)}^2\, .
$$
We consider again the function $g$ in \eq{gf} that we name here $h$ in order to distinguish their graphs, $h(t)=g(t)$ as in \eq{gf}.
Since $\lambda>\overline{\lambda}$, the coefficient of $(t^+)^2$ is now also negative and the qualitative graph of $h$ is as in the right
picture of Figure \ref{qualitative}. Then the function $h$ has a nondegenerate local maximum at $t=0$ which means that also the map
$t\mapsto J_0(t\overline{e}_1)$ has a local maximum at $t=0$ and it is strictly negative in a punctured interval containing $t=0$. Let
$E={\rm span}\{\overline{e}_k;\, k\ge2\}$ denote the infinite dimensional space of codimension $1$ being the orthogonal complement of
${\rm span}\{\overline{e}_1\}$. By the improved Poincar\'e inequality
$$
\lambda_2\|v_x\|_{L^2(\Omega)}^2\le\|v\|_{H^2_*(\Omega)}^2\qquad\forall v\in E
$$
and by taking into account Lemma \ref{lemmad} (i) and $\lambda\le\lambda_2$, we see that
$$
J_0(u)\ge\frac{\lambda_2-\lambda}{2\lambda_2}\|u\|_{H_*^2(\Omega)}^2+\int_\Omega{\Upsilon(y)\left(\frac{k}{2}(u^+)^2+\frac{\delta}{4}
(u^+)^4\right)}\ge0\qquad\forall u\in E\, .
$$
Therefore, the two open sets
$$A^+=\{u\in H^2_*(\Omega);\, (u,\overline{e}_1)_{H^2_*(\Omega)}>0,\, J_0(u)<0\}\, ,\quad
A^-=\{u\in H^2_*(\Omega);\, (u,\overline{e}_1)_{H^2_*(\Omega)}<0,\, J_0(u)<0\}$$
are disconnected. Since $J_0$ satisfies the (PS) condition and is bounded from below, $J_0$ admits a global minimum $u^+$ (resp.\ $u^-$) in $A^+$
(resp.\ $A^-$) and $J_0(u^\pm)<0$.\par
A sufficiently small linear perturbation of $J_0$ then has a local minimum in a neighborhood of both $u^\pm$. Whence, if $f$ is sufficiently
small, say $\|f\|_{L^2(\Omega)}<K$, then the functional $J$ defined by $J(u)=J_0(u)-\integ fu$ admits a local minimum in two neighborhoods of
both $u^\pm$. A minimax procedure then yields an additional (mountain-pass) critical point, see the proof of Theorem \ref{loadf} (iv) for the
details. This yields a third solution of \eq{finalmostro}.

\end{document}